\documentclass[11pt,fullpage, doublespace]{amsart}
\usepackage{amsfonts, amssymb, amsmath}
\usepackage[all]{xy}

\newtheorem{theorem}{Theorem}[section]

\newtheorem{defi}{Definition}[section]
\newtheorem{corr}{Corollary}[section]

\newcommand{\ap}{\alpha}

\newcommand{\gm}{\gamma}
\newcommand{\dt}{\delta}

\newcommand{\Z}{\mathbb{Z}}

\newcommand{\Q}{\mathbb{Q}}

\newcommand{\F}{\mathbb{F}}

\title{Average Results on the Order of $a$ modulo $p$}
\author{Kim, Sungjin}

\begin{document}
    \maketitle

    \begin{abstract}
    Let $a>1$ be an integer. Denote by $l_a(p)$  the multiplicative order of $a$ modulo primes $p$. We prove that
    if $\frac{x}{\log x\log\log x}=o(y)$, then
    $$
    \frac 1 y \sum_{a\leq y}\sum_{p\leq x}\frac{1}{l_a(p)}=\log x + C\log\log x+O\left(\frac x {y \log\log x}\right) $$
    which is an improvement over a theorem by Felix ~\cite{Fe}.

    Additionally, we also prove two other average results

    If $\log^2 x = o(\psi(x))$ and $x^{1-\dt}\log^3 x = o(y)$, then
    $$
    \frac1y\sum_{a<y}\sum_{\substack{{p<x}\\{l_a(p)>\frac{x}{\psi(x)}}}} 1 =  \pi(x)  + O\left(\frac{x\log x}{\psi(x)}\right)+O\left(\frac{x^{2-\dt}\log^2 x}y\right).$$
    Furthermore, if $x^{1-\dt}\log^3 x = o(y)$, then
    $$
    \frac1y\sum_{a<y}\sum_{\substack{{p<x}\\{p\nmid a}}}l_a(p) =  c\textrm{Li}(x^2) +O\left(\frac{x^2}{\log^A x}\right) + O\left(\frac{x^{3-\dt}\log^2 x}y\right)$$
    where
    $$
    c=\prod_p \left(1-\frac p{p^3-1}\right).$$
        \end{abstract}

\section{Introduction}
    Let $a>1$ be an integer. If $p$ be a prime not dividing $a$, we write  $d=l_a(p)$ if $d$ is the multiplicative order of $a$ modulo $p$. Then $d$ is the smallest positive integer in the congruence $a^d\equiv 1$ (mod $p$). Artin's Conjecture on Primitive Roots (AC) states that $l_a(p)=p-1$ for infinitely many primes $p$. Assuming the Generalized Riemann Hypothesis (GRH), Hooley ~\cite{Ho} proved that $l_a(p)=p-1$ for positive proportion of primes $p\leq x$. It is expected that $l_a(p)$ is large for majority of primes $p\leq x$. In ~\cite{EM}, Erdos and   Murty showed that $l_a(p) \geq p^{1/2+\epsilon(p)}$ for all but $o(\pi(x))$ primes $p\leq x$ where $\epsilon(p)\rightarrow 0$. With much simpler method, they showed  a weaker result $l_a(p)>\frac{\sqrt{p}}{\log p}$ for all but $O(x/\log^3 x)$ primes $p\leq x$. F. Pappalardi ~\cite{P} showed that there exist $\alpha, \delta >0$ such that $l_a(p)\geq p^{1/2}\exp{(\log^{\delta} p)}$ for all but $O(x/\log^{1+\ap} x)$. Kurlberg and Pomerance ~\cite{KP2} applied   Fouvry ~\cite{Fo} to show that there is $\gm>0$ such that $l_a(p)>p^{1/2+\gm}$ for positive proportion of primes $p\leq x$.

    Therefore, it is natural to expect that the average reciprocal of $l_a(p)$ is quite small. Murty and   Srinivasan ~\cite{MS} showed that $\sum_{p<x} \frac{1}{l_a(p)} =O(\sqrt x)$  and that $\sum_{p<x}\frac{1}{l_a(p)} = O(x^{1/4})$ implies AC for $a$.  F. Pappalardi ~\cite{P} proved that for some positive constant $\gamma$, $$\sum_{p<x} \frac{1}{l_a(p)} =O\left(\frac{\sqrt x}{\log^{1+\gamma}x}\right).$$
    For fixed $a$, it seems that it is very difficult to reduce $\sqrt x$ with current knowledge. However, we expect that averaging over $a$ would give some information. So, we take average over $a<y$, but we do not want to have too large $y$ such as $y>x$. For all the average result in this paper, we assume that $y<x$, and try to obtain $y$ as small as possible. The following result by Felix ~\cite{Fe} supports that $l_a(p)$ is mostly large:

    If $\frac{x}{\log x}=o(y)$, then
    $$
    \frac 1 y \sum_{a\leq y}\sum_{p\leq x}\frac{1}{l_a(p)}=\log x + O(\log\log x)+O\left(\frac x y \right).$$
    Felix remarked that the first error term $O(\log\log x)$ can be $C\log\log x + O(1)$ by applying Fiorilli's method ~\cite{Fi}, but did not explicitly find $C$. We find the $C$ in Theorem 1.1. This detailed estimate takes effect when $\frac x{(\log\log x)^2} = o(y)$. We apply a deep result on exponential sums by Bourgain ~\cite{B} to obtain Corollary 2.2 which will be the key for all average results in this paper.
    \begin{theorem}
    If $\frac{x}{\log x\log\log x}=o(y)$, then
    $$
    \frac 1 y \sum_{a\leq y}\sum_{p\leq x}\frac{1}{l_a(p)}=\log x + C\log\log x+O(1)+O\left(\frac x {y \log\log x}\right)  $$
    where
    \begin{align*}
    C&=  2\gamma-2\sum_{p}\frac{\log p}{p^2-p+1}  \\
    & \ \ +  \frac{\zeta(2)\zeta(3)}{\zeta(6)}\sum_{k=1}^{\infty} \frac{\mu(k)}{k^2} \left(-2\sum_{p|k}\frac{(p-1)p\log p}{p^2-p+1}+\log k\right)\prod\limits_{p|k} \left(1+\frac{p-1}{p^2-p+1}\right).
    \end{align*}
    \end{theorem}
    Assuming GRH for Kummer extensions $\Q(\zeta_d, a^{1/d})$ , F. Pappalardi ~\cite[Theorem 4.1]{P} proved that for increasing function $\psi(x)$ tending to infinity, $l_a(p)\geq \frac p{\psi(p)}$ for all but $O\left(\frac{\pi(x)\log \psi(x)}{\psi(\sqrt x)}\right)$ primes $p\leq x$. We prove that  unconditionally on average, a result similar to F. Pappalardi's theorem holds with restriction on $\psi$ function $\log^2 x = o(\psi(x))$.
    \begin{theorem}
    Let $\psi(x)$ be an increasing function such that $\log^2 x = o(\psi(x))$. Let $\dt$ be the positive constant in Corollary 2.4. If $x^{1-\dt}\log^3 x = o(y)$, then
    $$
    \frac1y\sum_{a<y}\sum_{\substack{{p<x}\\{l_a(p)>\frac{x}{\psi(x)}}}} 1 =  \pi(x)  + O\left(\frac{x\log x}{\psi(x)}\right)+O\left(\frac{x^{2-\dt}\log^2 x}y\right).$$
    \end{theorem}
    Assuming GRH for Kummer extensions $\Q(\zeta_d, a^{1/d})$, P. Kurlberg and C. Pomerance ~\cite{KP} showed that
    $$
    \frac{1}{\pi(x)} \sum_{p<x} l_2(p) = \frac{159}{320}cx + O\left(\frac{x}{{(\log x)}^{1-4/\log\log\log x}}\right)$$
    with $c=\prod_p \left(1-\frac p{p^3-1}\right)$. An average result over all possible nonzero residue classes is obtained by F. Luca ~\cite{L}:
    For any constant $A>0$,
    $$
    \frac{1}{\pi(x)} \sum_{p<x}\frac{1}{(p-1)^2}\sum_{a=1}^{p-1} l_a(p) = c + O\left(\frac 1{\log^A x}\right).
    $$
    By partial summation, this gives the following statistics on average order:
    $$\frac{1}{\pi(x)}\sum_{p<x}\frac{1}{p-1}\sum_{a=1}^{p-1} l_a(p) = \frac12cx + O\left(\frac{x}{\log x}\right).$$
    We prove that unconditionally on average, a similar result holds with average order $cp$.
    \begin{theorem}
    Let $A>0$ be any constant, and $\dt>0$ be the constant   in Corollary 2.4. If $x^{1-\dt}\log^3 x= o(y)$, then
    $$
    \frac1y\sum_{a<y}\sum_{\substack{{p<x}\\{p\nmid a}}}l_a(p) =  c\textrm{Li}(x^2) +O\left(\frac{x^2}{\log^A x}\right) + O\left(\frac{x^{3-\dt}\log^2 x}y\right)$$
    where
    $$
    c=\prod_p \left(1-\frac p{p^3-1}\right).$$
    \end{theorem}
    In the form of P. Kurberg and C. Pomerance's result, this is
    \begin{corr}
    Let  $\dt>0$ be the constant in Corollary 2.4. If $x^{1-\dt}\log^3 x = o(y)$, then
    $$
    \frac1y\sum_{a<y}\frac{1}{\pi(x)}\sum_{\substack{{p<x}\\{p\nmid a}}}l_a(p) =  \frac12c x +O\left(\frac{x }{\log  x}\right) + O\left(\frac{x^{2-\dt}\log^3 x}y\right).$$
    \end{corr}

   \section{Backgrounds}
    \subsection{Equidistribution}
    A sequence $\{a_n\}$ of real numbers are said to be equidistributed modulo $1$ if the following is satisfied:
    \begin{defi}
    Let $0\leq a<b\leq 1$. Suppose that
    $$
    \lim_{N\rightarrow\infty}\frac1N|\{n\leq N \ : \ a_n \in (a,b) \textrm{ mod $1$}\}| = b-a.$$
    Then we say that $\{a_n\}$ is equidistributed modulo $1$.
    \end{defi}
    A well-known criterion by Weyl ~\cite{W} is
    \begin{theorem}
    For any integer $k\neq 0$, suppose that
    $$
    \lim_{N\rightarrow\infty}\frac1N \sum_{n\leq N}e^{2\pi i ka_n } = 0.$$
    Then the sequence $\{a_n\}$ is equidistributed modulo $1$.
    \end{theorem}
    There were a series of effort to obtain the quantitative form of equidistribution theorem. Erd\"{o}s and Tur\'{a}n ~\cite{ET} succeed in obtaining such form:
    \begin{theorem}
    Let $\{a_n\}$ be a sequence of real numbers. Then
    $$
    \sup_{0\leq a<b\leq 1} \left||\{ n\leq N :  a_n \in (a,b) \textrm{ mod $1$} \}| - (b-a)N\right| \leq c_1 \frac{N}{M+1} + c_2 \sum_{m=1}^M\frac1m \left|\sum_{n\leq N}e^{2\pi i m a_n}\right|.
    $$
    \end{theorem}
    H. Montgomery ~\cite{M} obtained $c_1=1$, $c_2=3$. C. Manduit, J. Rivat, A. S\'{a}rk\H{o}zy ~\cite{MRS} obtained $c_1=c_2=1$. Thus, we have a quantitative upper bound of discrepancy when we have good upper bounds for exponential sums.
    \subsection{Exponential Sums in Prime Fields}
    J. Bourgain ~\cite{B} obtained the following equidistribution result for the subgroup $H<\F_p^{*}$ when $|H|>p^{\frac{C}{\log\log p}}$ for some absolute constant $C>1$ by sum-product method. See also ~\cite{BG}.
    \begin{theorem}
    Let $p$ be a prime. There exist  absolute constants $C>1$ and $C_1>0$ such that for any subgroup $H$ of $\F_p^{*}$ with $|H|>p^{\frac{C}{\log\log p}}$,
    $$
    \max_{(k,p)=1}\left|\sum_{a\in H}e^{2\pi i k\frac {a}p}\right|<e^{-\log^{C_1}p}|H|.$$
    \end{theorem}
    Since any subgroup $H$ of $\F_p^{*}$ is cyclic, we consider $|H|=d | p-1$. Then $H$ consists of all $d$-th roots of unity in $\F_p$. This yields
    \begin{corr}
    Let $p$ be a prime and $1\leq d | p-1$. Suppose that $d > p^{\frac{C}{\log\log p}}$. Then we have
    $$
    \max_{(m,p)=1}\left|\sum_{a\in\F_p, \   a^d= 1}  e^{2\pi i m\frac {a}p}\right|<de^{-\log^{C_1}p} .$$
    \end{corr}
    Combining this with Erd\"{o}s-Tur\'{a}n inequality, we obtain the following
    \begin{corr}
    Let $p$ be a prime, and $y\geq 1$. Assume that $d|p-1$ and $d > p^{\frac{C}{\log\log p}}$. Then for any constant $C_2>0$ smaller than $C_1$ in Corollary 2.1, we have
    $$
    \sum_{a<y, \ a^d\equiv 1 (p)} 1 = \frac yp d + O(de^{-\log^{C_2}p}).
    $$
    \end{corr}
    \begin{proof}
      Since $d|p-1$, the congruence $a^d\equiv 1$ yields $d$ roots in $\F_p$. Thus, we need to count $a<y$ satisfying those $d$ congruences modulo $p$. Considering $\frac yp = \lfloor \frac yp \rfloor + \frac yp - \lfloor \frac yp \rfloor$, it is enough to prove the result for $y<p$. We apply Erd\"{o}s-Tur\'{a}n inequality to the set $\{\frac ap : a^d = 1\}$. Then
    \begin{align*}
    \left||\{ 0\leq a\leq p-1 :  a^d\equiv 1 (p), \frac ap \in (0,\frac yp) \textrm{ mod $1$} \}| - \frac yp d\right| &\leq   \frac{d}{p} +   \sum_{m=1}^{p-1}\frac1m \left|\sum_{a \leq p-1 , \ a^d\equiv 1 (p)}e^{2\pi i m \frac ap}\right|\\
    &\leq \frac{d}{p} + (2\log p) de^{-\log^{C_1}p}\\
    &\leq d e^{-\log^{C_2}p}.
    \end{align*}
    This completes the proof.  \end{proof}

    For the Theorem 1.2 and 1.3, we need J. Bourgain's result when the subgroup $H$ has order greater than $p^{\epsilon}$ for fixed $\epsilon>0$.
    \begin{theorem}
    Let $p$ be a prime. For any fixed $\epsilon>0$, There exist  a constant $\dt=\dt(\epsilon)>0$ such that for any subgroup $H$ of $\F_p^{*}$ with $|H|>p^{\epsilon}$,
    $$
    \max_{(k,p)=1}\left|\sum_{a\in H}e^{2\pi i k\frac {a}p}\right|<p^{-\dt}|H|.$$
    \end{theorem}
    Similarly, we have the following corollary:
    \begin{corr}
    Let $p$ be a prime and $1\leq d | p-1$. Let $\epsilon>0$ be fixed. Suppose that $d > p^{\epsilon}$. Then there exists a constant $\dt=\dt(\epsilon)>0$ such that
    $$
    \max_{(m,p)=1}\left|\sum_{a\in\F_p, \   a^d= 1}  e^{2\pi i m\frac {a}p}\right|<dp^{-\dt} .$$
    \end{corr}
    We omit the proof of the following corollary because it is similar to that of Corollary 2.2.
    \begin{corr}
    Let $p$ be a prime, and $y\geq 1$. Let $\epsilon>0$ be fixed. Assume that $d|p-1$ and $d > p^{\epsilon}$. Then there exists $\dt=\dt(\epsilon)>0$ such that
    $$
    \sum_{a<y, \ a^d\equiv 1 (p)} 1 = \frac yp d + O(dp^{-\dt}).
    $$
    \end{corr}
    Corollary 2.2 and 2.4 play   key roles in proving Theorem 1.1, 1.2, and 1.3. Note that this is significantly better than the trivial bound when $p$ is large:
    $$
    \sum_{a<y, \ a^d\equiv 1 (p)} 1 = \frac yp d + O(d).
    $$
    \section{Proof of Theorems}
    \subsection{Proof of Theorem 1.1}
    Let $\epsilon=\frac{4C}{\log\log x}$ and consider the summation change:
    \begin{align*}
     \sum_{a\leq y}\sum_{p\leq x}\frac{1}{l_a(p)} &= \sum_{d<x}\frac 1d \sum_{\substack{{p\leq x} \\ {p\equiv 1 (d)}}} \sum_{\substack{{a\leq y} \\ {l_a(p)=d}}}1\\
     &=\sum_{d<x^{\epsilon}} + \sum_{x^{\epsilon}\leq d<x}\\
     &=\Sigma_1 + \Sigma_2
    \end{align*}
    First, we treat $\Sigma_1$ by trivial bound and Brun-Titchmarsh inequality:
    \begin{align*}
    \Sigma_1&=\sum_{d<x^{\epsilon}}\frac 1d \sum_{\substack{{p\leq x} \\ {p\equiv 1 (d)}}} \sum_{\substack{{a\leq y} \\ {l_a(p)=d}}}1 \\
    &=\sum_{d<x^{\epsilon}}\frac 1d \sum_{\substack{{p\leq x} \\ {p\equiv 1 (d)}}} \left(\phi(d)\frac yp + O(\phi(d))\right)\\
    &=\sum_{d<x^{\epsilon}}\frac 1d \sum_{\substack{{p\leq x} \\ {p\equiv 1 (d)}}} \phi(d)\frac yp + O(E_1),
    \end{align*}
    where
    \begin{align*}
    E_1&=\sum_{d<x^{\epsilon}}\frac 1d \sum_{\substack{{p\leq x} \\ {p\equiv 1 (d)}}} \phi(d) = \sum_{d<x^{\epsilon}}\frac {\phi(d)}d \sum_{\substack{{p\leq x} \\ {p\equiv 1 (d)}}} 1\\
    &=\sum_{d<x^{\epsilon}}\frac {\phi(d)}d \pi(x;d,1)\\
    &\ll \sum_{d<x^{\epsilon}}\frac {\phi(d)}d  \frac{x}{\phi(d)\log x} \ll \epsilon x.
    \end{align*}
    Thus,
    $$\Sigma_1 = \sum_{d<x^{\epsilon}}\frac 1d \sum_{\substack{{p\leq x} \\ {p\equiv 1 (d)}}} \phi(d)\frac yp + O(\epsilon x).$$
    Now, we treat $\Sigma_2$ by M\"{o}bius inversion and Corollary 2.2:
    \begin{align*}
    \Sigma_2&=\sum_{x^{\epsilon}\leq d<x}\frac1d \sum_{\substack{{p\leq x} \\ {p\equiv 1 (d)}}}  \sum_{\substack{{a\leq y} \\ {l_a(p)=d}}}1\\
    &=\sum_{x^{\epsilon}\leq d<x}\frac1d \sum_{\substack{{p\leq x} \\ {p\equiv 1 (d)}}} \sum_{d'|d} \mu\left(\frac{d}{d'}\right)\sum_{\substack{{a\leq y} \\ {a^{d'}\equiv 1(p)}}}1\\
    &=\sum_{x^{\epsilon}\leq d<x}\frac1d \sum_{\substack{{p\leq x} \\ {p\equiv 1 (d)}}} \sum_{\substack{{d'|d}   \\ {d'<p^{\frac{C}{\log\log p}}}}} \mu\left(\frac{d}{d'}\right)\sum_{\substack{{a\leq y} \\ {a^{d'}\equiv 1(p)}}}1+\sum_{x^{\epsilon}\leq d<x}\frac1d \sum_{\substack{{p\leq x} \\ {p\equiv 1 (d)}}} \sum_{\substack{{d'|d}   \\ {d'\geq p^{\frac{C}{\log\log p}}}}} \mu\left(\frac{d}{d'}\right)\sum_{\substack{{a\leq y} \\ {a^{d'}\equiv 1(p)}}}1\\
    &=\sum_{x^{\epsilon}\leq d<x}\frac1d \sum_{\substack{{p\leq x} \\ {p\equiv 1 (d)}}} \sum_{\substack{{d'|d}   \\ {d'<p^{\frac{C}{\log\log p}}}}} \mu\left(\frac{d}{d'}\right)\left(\frac yp d' + O(d')\right)\\
    & \ \ \ + \sum_{x^{\epsilon}\leq d<x}\frac1d \sum_{\substack{{p\leq x} \\ {p\equiv 1 (d)}}} \sum_{\substack{{d'|d}   \\ {d'\geq p^{\frac{C}{\log\log p}}}}} \mu\left(\frac{d}{d'}\right)\left(\frac yp d' + O(d' e^{-\log^{C_2} p})\right).
    \end{align*}
    Then we have
    \begin{align*}
    &\Sigma_1+\Sigma_2 \\
    &=\sum_{d<x^{\epsilon}}\frac 1d \sum_{\substack{{p\leq x} \\ {p\equiv 1 (d)}}} \phi(d)\frac yp + \sum_{x^{\epsilon}\leq d<x}\frac1d \sum_{\substack{{p\leq x} \\ {p\equiv 1 (d)}}} \sum_{\substack{{d'|d}   \\ {d'<p^{\frac{C}{\log\log p}}}}} \mu\left(\frac{d}{d'}\right) \frac yp d' + \sum_{x^{\epsilon}\leq d<x}\frac1d \sum_{\substack{{p\leq x} \\ {p\equiv 1 (d)}}} \sum_{\substack{{d'|d}   \\ {d'\geq p^{\frac{C}{\log\log p}}}}} \mu\left(\frac{d}{d'}\right) \frac yp d'\\
    & \ \ \ + O(E_1)+O(E_2)+O(E_3)\\
    &=\sum_{d<x}\frac {\phi(d)}d \sum_{\substack{{p\leq x} \\ {p\equiv 1 (d)}}} \frac yp + O(E_1)+O(E_2)+O(E_3).
    \end{align*}
    where
    \begin{align*}
    E_2&=\sum_{x^{\epsilon}\leq d<x}\frac1d \sum_{\substack{{p\leq x} \\ {p\equiv 1 (d)}}} \sum_{\substack{{d'|d}   \\ {d'<p^{\frac{C}{\log\log p}}}}} \left|\mu\left(\frac{d}{d'}\right)\right| d'
    \end{align*}
    and
    \begin{align*}
    E_3&=\sum_{x^{\epsilon}\leq d<x}\frac1d \sum_{\substack{{p\leq x} \\ {p\equiv 1 (d)}}} \sum_{\substack{{d'|d}   \\ {d'\geq p^{\frac{C}{\log\log p}}}}} \left|\mu\left(\frac{d}{d'}\right)\right| d'e^{-\log^{C_2} p}.
    \end{align*}
    Here, the term
    $$
    \sum_{d<x}\frac {\phi(d)}d \sum_{\substack{{p\leq x} \\ {p\equiv 1 (d)}}} \frac yp $$
    is the main term in ~\cite[Theorem 1.4]{Fe}. It is proven to be $y\log x + O(y\log\log x)$ in ~\cite[Theorem 1.4]{Fe} which will be shown to be $y\log x + Cy \log\log x + O(1)$ later.

    We treat $E_2$. Since $\pi(x;d,1)\ll \frac xd$, we have:
    \begin{align*}
    E_2&=\sum_{x^{\epsilon}\leq d<x}\frac1d \sum_{\substack{{p\leq x} \\ {p\equiv 1 (d)}}} \sum_{\substack{{d'|d}   \\ {d'<p^{\frac{C}{\log\log p}}}}} \left|\mu\left(\frac{d}{d'}\right)\right| d'\\
    &\ll \sum_{x^{\epsilon}\leq d<x}\frac1d \sum_{\substack{{p\leq x} \\ {p\equiv 1 (d)}}} \sum_{d'<p^{\frac{C}{\log\log p}}}d'\\
    &\ll \sum_{x^{\epsilon}\leq d<x}\frac1d \sum_{\substack{{p\leq x} \\ {p\equiv 1 (d)}}} p^{\frac{2C}{\log\log p}}\\
    &\ll x^{\frac{2C}{\log\log x}} \sum_{x^{\epsilon}\leq d<x}\frac1d  \pi(x;d,1)\\
    &\ll x^{\frac{2C}{\log\log x}} \sum_{x^{\epsilon}\leq d<x} \frac x{d^2}.
    \end{align*}
    Since $\sum_{d\geq x} \frac1{d^2} \ll \frac1x$, we have
    $$
    E_2\ll  x^{1+\frac{2C}{\log\log x}-\epsilon} \ll x^{1-\frac{\epsilon}2}.$$
    We are left with $E_3$. First, we have the following:
    \begin{align*}
    \sum_{\substack{{d'|d}   \\ {d'\geq p^{\frac{C}{\log\log p}}}}} \left|\mu\left(\frac{d}{d'}\right)\right| d'&\leq \sum_{d'|d} \left|\mu\left(\frac{d}{d'}\right)\right| d'\\
    &\leq d \prod_{p|d}\left(1+\frac1p\right)\\
    &=d \prod_{p|d} \frac{ 1+\frac1p }{1-\frac1p} \left(1-\frac1p\right)\\
    &\leq \phi(d) 3^{\omega(d)}
    \end{align*}
    where $\omega(d)$ is the number of distinct prime factors of $d$.

    Again by $\pi(x;d,1)\ll \frac xd$, we have
    \begin{align*}
    E_3&\ll \sum_{x^{\epsilon}\leq d<x}\frac1d \sum_{\substack{{p\leq x} \\ {p\equiv 1 (d)}}} \phi(d) 3^{\omega(d)} e^{-\log^{C_3} x}  \\
    & \ll \sum_{x^{\epsilon}\leq d<x}\frac1d \phi(d)3^{\omega(d)} \frac xde^{-\log^{C_3} x}
    \end{align*}
    By partial summation with $\sum_{d\leq t} 3^{\omega(d)} \ll t\log^2 t$,
    \begin{align*}
    E_3 &\ll \sum_{x^{\epsilon}\leq d<x} \frac{3^{\omega(d)}}d x e^{-\log^{C_3} x}\ll x(\log^3 x )e^{-\log^{C_3} x}\ll x e^{-\log^{C_4}x}.
    \end{align*}
    Combining these estimates, we have
    $$
    \sum_{a\leq y}\sum_{p\leq x}\frac{1}{l_a(p)} = \sum_{d<x}\frac {\phi(d)}d \sum_{\substack{{p\leq x}   \\ {p\equiv 1 (d)}}} \frac yp + O(\epsilon x)+O(x^{1 - \frac{\epsilon}2 })+O(x e^{-\log^{C_4}x}) $$
    with the first error term dominating the other two. Hence,
    $$
    \sum_{a\leq y}\sum_{p\leq x}\frac{1}{l_a(p)} = \sum_{d<x}\frac {\phi(d)}d \sum_{\substack{{p\leq x}   \\ {p\equiv 1 (d)}}} \frac yp + O\left(\frac{x}{\log\log x}\right).
    $$
    Following the proof of ~\cite[Theorem 1.4]{Fe}, we have
    \begin{align*}
    \sum_{d<x}\frac{\phi(d)}d\pi(x;d,1)&=\sum_{k<x}\frac{\mu(k)}k\sum_{\substack{{p\leq x} \\ {p\equiv 1 (k)}}}\tau\left(\frac{p-1}k\right) \\
    &=\sum_{k\leq \log^{A+2} x} + \sum_{\log^{A+2}x<k<x}\\
    &=\sum_{k\leq \log^{A+2} x} + O\left(\frac{x}{\log^A x}\right).
    \end{align*}
    As Fiorilli and Felix pointed out, we apply
    $$
     \sum_{\substack{{p\leq x}\\{p\equiv 1(k)}}}\tau\left(\frac{p-1}{k}\right)=\frac{x}{k}C_1(k) + \frac{1}{k}\left(2C_2(k)+C_1(k)\log\left(\frac{(k')^2}{k}\right)\right)\textrm{li}(x)+O\left(\frac{x}{\log^A x}\right)$$
    where
    $$
    C_1(k)=\frac{\zeta(2)\zeta(3)}{\zeta(6)} \prod\limits_{p|k} \left(1+\frac{p-1}{p^2-p+1}\right),$$
    $$C_2(k)=C_1(k)\left(\gamma-\sum_{p}\frac{\log p}{p^2-p+1} -\sum_{p|k}\frac{(p-1)p\log p}{p^2-p+1}\right),$$
    and
    $k'=\prod_{p|k}p$.

    As in ~\cite[Theorem 1.4]{Fe}, all the sums over $k$ are absolutely convergent and $\sum \frac{\mu(k)C_1(k)}{k^2}=1$, so we have
    \begin{align*}
    \sum_{k\leq \log^{A+2} x}\frac{\mu(k)}k\sum_{\substack{{p\leq x} \\ {p\equiv 1 (k)}}}\tau\left(\frac{p-1}k\right)&= \sum_{k\leq\log^{A+2} x} \frac{\mu(k)}{k^2} \left(  x C_1(k) +  \left(2C_2(k)+C_1(k)\log\left(\frac{(k')^2}{k}\right)\right)\textrm{li}(x)\right)\\& \ \ +O\left(\frac{x}{\log^A x}\right)\\
    &= x + \left(2\gamma-2\sum_{p}\frac{\log p}{p^2-p+1}\right) \textrm{li}(x) \\ & \ \ + \left(\sum_{k=1}^{\infty} \frac{\mu(k)C_1(k)}{k^2}\left(-2\sum_{p|k}\frac{(p-1)p\log p}{p^2-p+1}+\log\left(\frac{(k')^2}{k}\right)\right)\right) \textrm{li}(x) \\& \ \ + O\left(\frac x{\log^A x}\right).
    \end{align*}
    Since $\textrm{li}(u)= \frac{u}{\log u} + O\left(\frac{u}{\log^2 u}\right)$, we finally obtain
    \begin{align*}
    \sum_{d<x}\frac {\phi(d)}d \sum_{\substack{{p\leq x}   \\ {p\equiv 1 (d)}}} \frac 1p &= \int_2^x\frac1{u^2}\sum_{k\leq u}\frac{\phi(k)}k\pi(x;k,1)du +O(1)\\
    &=\int_2^x \frac{1}{u^2} \left( u + C \frac{u}{\log u} + O\left(\frac{u}{\log^2 u}\right)\right) du +O(1)\\
    &=\log x + C \log\log x + O(1)
    \end{align*}
    where
    \begin{align*}
    C&=  2\gamma-2\sum_{p}\frac{\log p}{p^2-p+1}  \\
    & \ \ +  \sum_{k=1}^{\infty} \frac{\mu(k)C_1(k)}{k^2}\left(-2\sum_{p|k}\frac{(p-1)p\log p}{p^2-p+1}+\log\left(\frac{(k')^2}{k}\right)\right).
    \end{align*}
    Since the terms in the second sum over $k$ only appears when $k$ is square free, we have $k'=k$. Thus,
    \begin{align*}
    C&=  2\gamma-2\sum_{p}\frac{\log p}{p^2-p+1}  \\
    & \ \ +  \sum_{k=1}^{\infty} \frac{\mu(k)C_1(k)}{k^2}\left(-2\sum_{p|k}\frac{(p-1)p\log p}{p^2-p+1}+\log k \right).
    \end{align*}
    This completes the proof of Theorem 1.1.
    \subsection{Proof of Theorem 1.2}
    Let $\psi(x)$ be an increasing function which tends to infinity as $x\rightarrow\infty$. The rate of increase of $\psi(x)$ is to be determined. We start with the change of order in summation:
    \begin{align*}
    \sum_{a<y}\sum_{\substack{{p<x}\\{l_a(p)>\frac{x}{\psi(x)}}}} 1 &=\sum_{\frac{x}{\psi(x)}<d<x} \sum_{\substack{{p<x}\\{p\equiv 1 (d)}}}\sum_{\substack{{a<y}\\{l_a(p)=d}}} 1\\
    &=\sum_{\frac{x}{\psi(x)}<d<x}  \sum_{\substack{{p<x}\\{p\equiv 1 (d)}}}\sum_{\substack{{d'|d}\\{d'<p^{\epsilon}}}}\mu\left(\frac{d}{d'}\right)\sum_{\substack{{a<y}\\{a^{d'}\equiv 1 (p)}}}1 + \sum_{\frac{x}{\psi(x)}<d<x} \sum_{\substack{{p<x}\\{p\equiv 1 (d)}}}\sum_{\substack{{d'|d}\\{d'\geq p^{\epsilon}}}}\mu\left(\frac{d}{d'}\right)\sum_{\substack{{a<y}\\{a^{d'}\equiv 1 (p)}}}1\\
    &=\sum_{\frac{x}{\psi(x)}<d<x} \sum_{\substack{{p<x}\\{p\equiv 1 (d)}}} \sum_{\substack{{d'|d}\\{d'<p^{\epsilon}}}}\mu\left(\frac{d}{d'}\right)\left(\frac yp d' + O(d')\right)\\
    & \ \ + \sum_{\frac{x}{\psi(x)}<d<x} \sum_{\substack{{p<x}\\{p\equiv 1 (d)}}} \sum_{\substack{{d'|d}\\{d'\geq p^{\epsilon}}}}\mu\left(\frac{d}{d'}\right)\left(\frac yp d' + O(d'p^{-\dt})\right)\\
    &=\sum_{\frac{x}{\psi(x)}<d<x} \sum_{\substack{{p<x}\\{p\equiv 1 (d)}}} \frac yp \phi(d) + O(E_1)+O(E_2)
    \end{align*}
    where
    \begin{align*}
    E_1 &= \sum_{d<x}\sum_{\substack{{p<x}\\{p\equiv 1 (d)}}}\sum_{\substack{{d'|d}\\{d'<p^{\epsilon}}}}\left|\mu\left(\frac{d}{d'}\right)\right|d'\\
    &\ll \sum_{d<x}\sum_{\substack{{p<x}\\{p\equiv 1 (d)}}} \sum_{d'<p^{\epsilon}} d'\\
    &\ll x^{2\epsilon}\sum_{d<x} \pi(x;d,1)\\
    &\ll x^{1+2\epsilon}\log x
    \end{align*}
    and
    \begin{align*}
    E_2 &= \sum_{d<x}\sum_{\substack{{p<x}\\{p\equiv 1 (d)}}}\sum_{\substack{{d'|d}\\{d'\geq p^{\epsilon}}}}\left|\mu\left(\frac{d}{d'}\right)\right|d'p^{-\dt}\\
    &\ll \sum_{d<x}\sum_{\substack{{p<x}\\{p\equiv 1 (d)}}}\phi(d)3^{\omega(d)} p^{-\dt}\\
    &\ll \sum_{d<x}\phi(d)3^{\omega(d)}\frac{x^{1-\dt}}d\\
    &\ll x^{2-\dt}\log^2 x .
    \end{align*}
    Now we treat the main term. Since we have $\sum_{\substack{{p<x}\\{p\equiv 1 (d)}}} \frac1p = \frac{\log\log x + O(\log d)}{\phi(d)}$ by ~\cite[Lemma 2.5]{EP},
    \begin{align*}
    \sum_{\frac{x}{\psi(x)}<d<x} \phi(d) \sum_{\substack{{p<x}\\{p\equiv 1 (d)}}} \frac1p &= \sum_{d<x} \phi(d) \sum_{\substack{{p<x}\\{p\equiv 1 (d)}}}\frac1p - \sum_{d\leq \frac{x}{\psi(x)}} \phi(d) \sum_{\substack{{p<x}\\{p\equiv 1 (d)}}} \frac1p \\
    &=\sum_{p<x} \frac1p\sum_{d|p-1}\phi(d) - \sum_{d\leq \frac{x}{\psi(x)}} \phi(d) \sum_{\substack{{p<x}\\{p\equiv 1 (d)}}} \frac1p \\
    &=\sum_{p<x} \frac{p-1}p + O\left( \sum_{d\leq \frac{x}{\psi(x)}} \phi(d) \frac{\log\log x + \log d}{\phi(d)}\right)\\
    &=\pi(x)+O(\log\log x) +O\left(\frac{x\log x}{\psi(x)}\right).
    \end{align*}
    Combining all the estimates, we have
    $$
    \sum_{a<y}\sum_{\substack{{p<x}\\{l_a(p)>\frac{x}{\psi(x)}}}} 1 = y\pi(x)+O(y\log\log x) + O\left(\frac{xy\log x}{\psi(x)}\right) + O(x^{2-\dt}\log^2 x).
    $$
    Since we have $y<x$, the error term $O(y\log\log x)$ is dominated by $O(x^{2-\dt}\log^2 x)$.
    This completes the proof of Theorem 1.2.
    \subsection{Proof of Theorem 1.3}
    We begin with an application of Mobius inversion and Corollary 2.4:
    \begin{align*}
    \sum_{a<y}\sum_{d<x}\sum_{\substack{{p<x}\\{l_a(p)=d}}}d &=\sum_{d<x}d\sum_{\substack{{p<x}\\{p\equiv 1 (d)}}}\sum_{\substack{{a<y}\\{l_a(p)=d}}}1\\
    &=\sum_{d<x}d\sum_{\substack{{p<x}\\{p\equiv 1 (d)}}}\sum_{\substack{{d'|d}\\{d'<p^{\epsilon}}}}\mu\left(\frac{d}{d'}\right)\sum_{\substack{{a<y}\\{a^{d'}\equiv 1 (p)}}}1+\sum_{d<x}d\sum_{\substack{{p<x}\\{p\equiv 1 (d)}}}\sum_{\substack{{d'|d}\\{d'\geq p^{\epsilon}}}}\mu\left(\frac{d}{d'}\right)\sum_{\substack{{a<y}\\{a^{d'}\equiv 1 (p)}}}1\\
    &=\sum_{d<x}d\sum_{\substack{{p<x}\\{p\equiv 1 (d)}}}\sum_{\substack{{d'|d}\\{d'<p^{\epsilon}}}}\mu\left(\frac{d}{d'}\right)\left(\frac yp d' + O(d')\right)\\ &
 \ \ +\sum_{d<x}d\sum_{\substack{{p<x}\\{p\equiv 1 (d)}}}\sum_{\substack{{d'|d}\\{d'\geq p^{\epsilon}}}}\mu\left(\frac{d}{d'}\right)\left(\frac yp d' + O(d'p^{-\dt})\right)\\
 &=\sum_{d<x}d\phi(d)\sum_{\substack{{p<x}\\{p\equiv 1 (d)}}}\frac yp + O(E_1)+O(E_2),
    \end{align*}
    where
    \begin{align*}
    E_1 &= \sum_{d<x}d\sum_{\substack{{p<x}\\{p\equiv 1 (d)}}}\sum_{\substack{{d'|d}\\{d'<p^{\epsilon}}}}\left|\mu\left(\frac{d}{d'}\right)\right|d'\\
    &\ll \sum_{d<x}d\sum_{\substack{{p<x}\\{p\equiv 1 (d)}}} \sum_{d'<p^{\epsilon}} d'\\
    &\ll x^{2\epsilon}\sum_{d<x} d\pi(x;d,1)\\
    &\ll x^{2+2\epsilon}
    \end{align*}
    and
    \begin{align*}
    E_2 &= \sum_{d<x}d\sum_{\substack{{p<x}\\{p\equiv 1 (d)}}}\sum_{\substack{{d'|d}\\{d'\geq p^{\epsilon}}}}\left|\mu\left(\frac{d}{d'}\right)\right|d'p^{-\dt}\\
    &\ll \sum_{d<x}d\sum_{\substack{{p<x}\\{p\equiv 1 (d)}}}\phi(d)3^{\omega(d)} p^{-\dt}\\
    &\ll \sum_{d<x}d\phi(d)3^{\omega(d)}\frac{x^{1-\dt}}d\\
    &\ll x^{3-\dt}\log^2 x.
    \end{align*}
    Now we treat the main term:
    \begin{align*}
    \sum_{d<x}d\phi(d)\sum_{\substack{{p<x}\\{p\equiv 1 (d)}}}\frac yp&=y\sum_{p<x}\frac1p\sum_{d|p-1} d\phi(d)\\
    &=y \sum_{p<x} \frac{(p-1)\ap(p-1)}p  \\
    &=y \left(\sum_{p<x} \ap(p-1) - \sum_{p<x} \frac{\ap(p-1)}p\right)
    \end{align*}
    Here, $\alpha(n) = \frac1n \sum_{d|n} d\phi(d)$ is the average order of $\Z/n\Z$. We use the following theorem by F. Luca ~\cite[Theorem 1]{L}:
    \begin{theorem}
    For any constant $A>0$,
    $$
    \frac{1}{\pi(x)} \sum_{p<x}\frac{\ap(p-1)}{p-1} = c + O\left(\frac 1{\log^A x}\right)
    $$
    where
    $$
    c= \prod_{p} \left(1-\frac p{p^3-1}\right).$$
    \end{theorem}
    Applying this theorem with partial summation, we obtain
    \begin{align*}
    \sum_{p<x} \ap(p-1) - \sum_{p<x} \frac{\ap(p-1)}p&=c \ \textrm{Li}(x^2) + O\left( \frac{x^2}{\log^A x}\right).
    \end{align*}
    Therefore,
    $$
    \sum_{a<y}\sum_{d<x}d\sum_{\substack{{p<x}\\{l_a(p)=d}}}1 =c y \textrm{Li}(x^2) + O\left( \frac{yx^2}{\log^A x}\right) + O(x^{3-\dt}\log^2 x).$$
    This completes the proof of Theorem 1.3.

    For the proof of Corollary 1.1, we use $\textrm{Li}(x^2) = \frac12 x\pi(x) + O\left(\frac{x^2}{\log^2 x}\right)$.

    \section{Remarks}
    The theorems in this paper have resemblance. If we change order of summation to put $\sum_d$ first, Theorem 1.1 is essentially $\sum_d d^{-1}\sum_p\sum_a$. Theorem 1.2 is   $\sum_d d^{0}\sum_p\sum_a$, and Theorem 1.3 is   $\sum_d d^{1}\sum_p\sum_a$.
    There is a difference in the method of Theorem 1.1, and the other two. In Theorem 1.1, we split the sum into four parts, while we split into three parts in Theorem 1.2 and Theorem 1.3.  This is because $d^{-1}$ is large for small $d$'s. We do not have a better information than $O(\epsilon x)$ for the error term $O(E_1)$ in Theorem 1.1, unless we have better exponential sum results. However, the method presented in this paper has wide variety of applications. For various conditional results, we could obtain the corresponding unconditional average results, and this method of exponential sums is powerful in shortening the range of averaging. In the upcoming paper, we will consider problems on the order of $a$ modulo $n$, for general modulus $n$.
        \flushleft

\end{document}